\documentclass[12pt]{article}
\usepackage[utf8]{inputenc}
\usepackage{hyperref}  
\usepackage{amsmath,amssymb,amsthm}
\usepackage{xcolor}
\usepackage[normalem]{ulem}  
\renewcommand\le{\leqslant}
\renewcommand\ge{\geqslant}

\newcommand\R{\mathbb R}
\newcommand\eps{\varepsilon}

\DeclareMathOperator\supp{supp}

\newtheorem*{theorem}{Theorem}
\newtheorem{lemma}{Lemma}

\newtheorem*{example}{Example}

\newtheorem{exttheorem}{Theorem}

\newtheorem{extstatement}{Statement}

\title{Kolmogorov widths of balls in mixed norms: the case of rigidity}
\author{Yuri Malykhin\thanks{Steklov Mathematical Institute, Lomonosov Moscow
State University; malykhin-yuri@yandex.ru},
Konstantin Ryutin\thanks{Lomonosov Moscow State University, Moscow Center of
Fundamental and Applied Mathematics; kriutin@yahoo.com}}

\begin{document}

\maketitle

\begin{abstract}
We describe the set of parameters $(p_1,p_2,q_1,q_2)$ such that the balls
$B_{q_1,q_2}^{s,b}$ are rigid in $\ell_{q_1,q_2}^{s,b}$ metric i.e. they are
poorly approximated by linear subspaces of dimension $\le (1-\eps)sb$, for
large $s,b$. 
Thus we have settled an important qualitative case in the problem of estimating
widths of balls in mixed norms.  The proof combines lower bounds from our
previous papers  and a new construction for the approximation by linear
subspaces in the so-called
exceptional case.
\end{abstract}

\paragraph{Introduction.}

A fundamental approximation characteristic is the Kolmogorov width~--- the error
of approximation of a set in a normed space by subspaces of given dimension. The
discretization technique often allows to estimate the widths of functional
classes via the widths of the corresponding finite-dimensional sets.
E.M.~Galeev initiated the systematic study of widths of balls in mixed norms
in order to estimate the widths of some multivariate function classes.

Let us say (informally) that an $N$-dimensional set is \textit{rigid} if it can
not be well approximated by linear spaces of dimension essentially smaller than
$N$. A typical example of an extremely rigid set is
an $N$-dimensional section of the unit ball of a normed
space: Tikhomirov's theorem states that any subspace of
dimension $<N$ does not approximate the ball better than a single point $0$.

We consider the Kolmogorov widths of finite-dimensional balls in mixed norms
and focus on the rigid case.
This paper continues the
study of the rigidity conducted in the series of
papers~\cite{M22}, \cite{M24}, \cite{MR25}.

We recall some definitions. The error of approximation of some set $K$ in a
normed space $X$ by subspaces of given dimension is captured by the quantity
called the Kolmogorov width of a set $K\subset
X$:
$$
d_n(K,X) := \inf_{Q_n\subset X}\sup_{x\in
K}\inf_{y\in Q_n}\|x-y\|_X,
$$
where  $\inf_{Q_n}$ is over all linear subspaces in $X$ of dimension at most $n$. 
We will use some well known estimates for the widths of finite-dimensional balls
(see e.g.~\cite{LGM,Tikh87,P85}).

Let $\ell_q^N$ denote the space $\R^N$ with the norm
$\|x\|_q := (\sum_{k=1}^N |x_k|^q)^{1/q}$ and
$B_q^N$ be its  unit ball.

It is well--known (see~\cite{Glus83}) that if $q>\max\{p,2\}$ then the ball
$B_p^N$ is well approximated by low-dimensional spaces:
\begin{equation}
    \label{lp_ball_nonrigid}
            d_n(B_p^N,\ell_q^N) \le C(p,q)N^{-\delta},
        \quad\mbox{if $n\ge N^{1-\delta},\;\delta=\delta(p,q)>0$}.
    \end{equation}
However, if $q\le\max\{p,2\}$, then \sout{all} the balls $B_p^N$ are
rigid in $\ell_q^N$ in the following sense:
\begin{equation}
    \label{lp_ball_rigidity}
d_n(B_p^N,\ell_q^N) \ge c(\eps)d_0(B_p^N,\ell_q^N) = c(\eps)\sup_{x\in B_p^N}\|x\|_q,
\quad \mbox{if }n\le N(1-\eps).
\end{equation}
Indeed, if $p\ge q$, then (see~\cite{Piet74})
\begin{equation}
    \label{pietch_stesin}
d_n(B_p^N,\ell_q^N) = (N-n)^{\frac1q-\frac1p} \ge
    \eps^{\frac1q-\frac1p}N^{\frac1q-\frac1p} =
    \eps^{\frac1q-\frac1p}d_0(B_p^N,\ell_q^N).
\end{equation}
If $p<q\le 2$, then $d_0(B_p^N,\ell_q^N)=1$, but
$d_n(B_p^N,\ell_q^N) \ge d_n(B_1^N,\ell_2^N) = (1-n/N)^{1/2} \ge \eps^{1/2}$.

E.D.~Gluskin~\cite{G87} proved that the set $V_k^N = B_\infty^N \cap kB_1^N$ is
also rigid: $d_n(V_k^N,\ell_q^N) \ge c(q)k^{1/q}$, if $n\le N/2$ and $1\le q\le 2$.
In our previous paper~\cite{MR25} we generalized these ideas of Gluskin and
obtained the following result.

Suppose that $H$ is a subgroup of $S_N$. It acts in a usual way on the set
$\{1,\ldots,N\}$. Recall that the action is called transitive if for any
$i,j\in\{1,\ldots,N\}$ there is $h\in H$ such that $h(i)=j$. 
\begin{exttheorem}[\cite{MR25}]
\label{th_transit}
    Let $H$ be a group that transitively acts on the set of $N$
    coordinates. Let $K$ be an unconditional set in $\R^N$ that is also
    invariant under the action of $H$, i.e. for all $h\in H$,
    $$
    (x_1,\ldots,x_N)\in K \quad \Rightarrow \quad (\pm x_{h(1)},\ldots, \pm
    x_{h(N)})\in K.
    $$
    Then for $1\le q\le 2$, any $\eps\in(0,1)$ we have
    $$
    d_n(K,\ell_q^N) \ge c(q,\eps)\sup_{x\in K}\|x\|_q,\quad\mbox{if $n\le
    N(1-\eps)$.}
    $$
\end{exttheorem}

We recall the definition  of a mixed norm (the mix of the $\ell_{q_1}^s$ and
$\ell_{q_2}^{b}$ norms). Consider the space
$\R^N$ and suppose that its coordinates are divided
into $b$ blocks of size $s$ (therefore, $N=sb$). Let $x[j]\in \R^s$ denote the $j$-th block of
$x\in\R^N$. Define the norm of a vector $x\in\R^N$ as
$$
\|x\|_{q_1,q_2} := \|y\|_{q_2},
\quad \mbox{where }y:=(\|x[1]\|_{q_1},\ldots,\|x[b]\|_{q_1}).
$$
Or, in a straightforward way,
$$
\|x\|_{q_1,q_2} = 
(\sum_{j=1}^b (\sum_{s(j-1)<i\le sj}|x_i|^{q_1})^{q_2/q_1})^{1/q_2}
$$
(with the usual modification for $q_1=\infty$ or $q_2=\infty$).
The space with
this norm is denoted as $\ell_{q_1,q_2}^{s,b}$ and its unit ball is
$B_{q_1,q_2}^{s,b}$.
The study of widths in
mixed norms was initiated by E.M.~Galeev in~\cite{Gal90,Gal95}
with connection to widths of multivariate function classes.
(Some authors prefer the notation
$\ell_{q_2}^b(\ell_{q_1}^s)$ but we decided to keep the notation 
introduced by Galeev.)

The orders of Kolmogorov widths $d_n(B_p^N,\ell_q^N)$, $1\le
p,q\le\infty$, are mostly known (with the exception of the
case $1\le p<2$, $q=\infty$). The situation with the orders of decay for
$d_n(B_{p_1,p_2}^{s,b},\ell_{q_1,q_2}^{s,b})$ is much less clear. 
We describe below the cases when these orders are known.
Let us refer to particular cases by specifying
quadruples $(p_1,p_2,q_1,q_2)$.

Most of the results correspond to the ``rigid'' case;
namely, the bounds
$d_n(B_{p_1,p_2}^{s,b},\ell_{q_1,q_2}^{s,b}) \asymp
d_0(B_{p_1,p_2}^{s,b},\ell_{q_1,q_2}^{s,b})$ for $n\le N/2$ were proved.
The first results were obtained by Galeev: \cite{Gal90}: $(1,\infty,2,q)$, $1<q<\infty$;
\cite{Gal95}: $(p_1,\infty,q_1,q_2)$, where $p_1\in\{1,\infty\}$; $1<q_2\le\infty$; $q_1=2$ or
$1<q_1\le\min\{q_2,2\}$.
Later, in~\cite{Gal01} he added the case $(\infty,1,q,q)$, $1<q\le 2$.
Galeev asked about the important case
$(1,\infty,2,1)$: A.D.~Izaak~\cite{Iza94} got a lower bound with a logarithmic gap;
Yu.V.~Malykhin and K.S.~Ryutin removed it in the paper~\cite{MR17};
a corollary is the rigidity for $p_1\le q_1\le 2$, $p_2\ge q_2$.
Izaak~\cite{Iza96} also mentions $(p,p,2,1)$, $p\in\{1\}\cup[2,\infty]$.
In~\cite{MR25} we obtained $(p_1,p_2,q,q)$, $1\le q\le 2$, see
Statement~\ref{stm_mixed_lp} below.
A.A.~Vasilieva~\cite{Vas24b} obtained the case $p_1\ge q_1$,
$p_2\ge q_2$.

The orders of widths are known also in several other cases.
Vasilieva~\cite{Vas13,Vas24b} settled the case $2\le q_1,q_2<\infty$.
S.~Dirksen and T.~Ullrich~\cite{DU18} considered the Gelfand widths
in mixed norms with the focus on the semi-normed case
($0<p\le 1$); by duality the orders of the widths were obtained
for: a) $(2,p,2,\infty)$, $p\ge 2$; b)
$(p,p,\infty,p)$, $p\ge 2$.

These estimates for $d_n(B_{p_1,p_2}^{s,b},\ell_{q_1,q_2}^{s,b})$ were applied
in the study of approximation properties of functional classes such as
Holder--Nikolski classes (see~\cite{Gal90,Gal96,Iza96,MR17}),
Besov classes of  mixed smoothness (\cite{Gal01,DU18}),
weighted Besov classes~\cite{Vas13};
see also a thorough survey on mixed smoothness
classes:~\cite{DTU}.
Mixed norms appear in the study of widths of Schatten classes,
see~\cite{HPV21,PS22}.
For the overview of the results on the widths of the balls in mixed norms,
see~\cite{Gal11,Vas24a}.
   
In this paper we settle the following subproblem~--- an important qualitative
case of the general problem of estimating
$d_n(B_{p_1,p_2}^{s,b},\ell_{q_1,q_2}^{s,b}$): describe the set of
parameters $p_1,p_2,q_1,q_2$ such that the balls $B_{p_1,p_2}^{s,b}$ are rigid
in the $\ell_{q_1,q_2}^{s,b}$ norm for large $s,b$.

\begin{theorem}
   Let $1\le p_1,p_2,q_1,q_2\le\infty$; $s,b\in\mathbb N$ and $N:=sb$.

    A) Assume that the tuple $(p_1,p_2,q_1,q_2)$
    satisfies three conditions:
    (i) $q_1\le\max\{p_1,2\}$; (ii) $q_2\le\max\{p_2,2\}$;
   (iii) it does
    not fall into the ``exceptional case''
    \begin{equation}
        \label{exception}
        q_1<\min\{p_1,q_2\}\mbox{ and }p_2<q_2\le 2.
    \end{equation}
    Then for any $\eps>0$ we have
    $$
        d_n(B_{p_1,p_2}^{s,b},\ell_{q_1,q_2}^{s,b})
        \ge c(q_2,\eps)d_0(B_{p_1,p_2}^{s,b},\ell_{q_1,q_2}^{s,b}),
        \quad\mbox{if $n\le N(1-\eps)$.}
    $$
    B) For any tuple $(p_1,p_2,q_1,q_2)$ that does not satisfy A), there exists
       $\gamma>0$, such that for sufficiently large  $t:=\min(s,b)$  we have
    $$
    d_n(B_{p_1,p_2}^{s,b},\ell_{q_1,q_2}^{s,b})
    \le t^{-\gamma} d_0(B_{p_1,p_2}^{s,b},\ell_{q_1,q_2}^{s,b}),\quad\mbox{when
    $n\ge Nt^{-\gamma}$}.
    $$
\end{theorem}

We have already used that $d_0(B_p^N,\ell_q^N)=N^{(1/q-1/p)_+}$, where, as
usual, $h_+:=h$, if $h\ge 0$, and $h_+:=0$ if $h<0$. Therefore, for the mixed
norms, we have
$$
d_0(B_{p_1,p_2}^{s,b},\ell_{q_1,q_2}^{s,b}) =
s^{(\frac1{q_1}-\frac1{p_1})_+}b^{(\frac1{q_2}-\frac1{p_2})_+}.
$$

\begin{example}
    Let us give an illustrative example of an ``exceptional
    case'': $d_n(B_{\infty,1}^{s,b},\ell_{1,2}^{s,b})$. We treat  $x$ as an
    $s\times b$ matrix; the block $x[j]$ as the
    $j$-th column of $x$.  It is clear that the
    extremal points of $B_{\infty,1}^{s,b}$ are matrices having support in one
    column, with $\pm1$ entries.
    Suppose that $s=b$ and denote by $x^T$ the transpose of the matrix $x$.
    For any extremal point $x$ we have
    
    $$
    \|x-(x-x^T)\|_{1,2} = \|x^T\|_{1,2} \le s^{1/2},
    $$
        since $x^T$ has only one nonzero entry in each column.
    Therefore, $x$ is approximated by a skew--symmetric matrix $x-x^T$  (these matrices are elements of $s(s-1)/2$--dimensional subspace)  with an
    error $s^{1/2}$ and we obtain
    $$
    d_{s(s-1)/2}(B_{\infty,1}^{s,s},\ell_{1,2}^{s,s}) \le s^{1/2},
    $$
    while $d_0(B_{\infty,1}^{s,s},\ell_{1,2}^{s,s}) = s$.
\end{example}

    We see that the mix of two rigid sets may be not rigid in the mixed norm.

    Part A) of Theorem is a simple corollary of the following two results.
    The first one is from $\cite{MR25}$.
\begin{extstatement}
    \label{stm_mixed_lp}
    Let $1\le p_1,p_2\le\infty$; $1\le q\le 2$; $N=sb$. Then for $n\le
    N(1-\eps)$, $\eps\in(0,1)$, we have
    $$
    c(q,\eps)s^{(\frac1q-\frac1{p_1})_+}b^{(\frac1q-\frac1{p_2})_+}
    \le d_n(B_{p_1,p_2}^{s,b},\ell_q^N)
    \le s^{(\frac1q-\frac1{p_1})_+}b^{(\frac1q-\frac1{p_2})_+}.
    $$
\end{extstatement}

Statement~\ref{stm_mixed_lp} follows from Theorem~\ref{th_transit}, because the ball
$B_{p_1,p_2}^{s,b}$ is invariant under the action of the
group $S_s\times S_b$. The group acts on the coordinates in the following way: $S_b$ permutes
blocks and $S_s$ permutes coordinates in blocks. (The proof given in~\cite{MR25}
is slightly different.)

The second one was proven in~\cite{MR17}.
\begin{extstatement}
    $d_{N/2}(B_{1,\infty}^{s,b},\ell_{2,1}^{s,b})\ge cb$.
\end{extstatement}
In fact, we need slightly more general inequality, namely,
\begin{equation}
    \label{oct_l21}
    d_n(B_{1,\infty}^{s,b},\ell_{2,1}^{s,b})\ge c(\eps)b,
    \quad\mbox{if $n\le N(1-\eps)$.}
\end{equation}
We checked that the proof from~\cite{MR17} needs only slight changes and gives the required estimate  for arbitrary $\eps>0$.

The rest of the paper is devoted to the proof of the Theorem.

\paragraph{A) Lower bounds.}

Each of the conditions (i), (ii) breaks into two non-overlapping cases: $p_i\ge q_i$ or $p_i<q_i\le
2$; we will consider them separately.

Case a) $p_1\ge q_1$, $p_2\ge q_2$. We have
$d_0=s^{(1/q_1-1/p_1)}b^{(1/q_2-1/p_2)}$. Since we have the inclusion 
$B_{p_1,p_2}^{s,b}\supset s^{-1/p_1}b^{-1/p_2}B_\infty^N$,
it is sufficient to prove the rigidity for the cube, $p_1=p_2=\infty$.
If $q_1\le q_2$, then by~\eqref{pietch_stesin} we have
$$
d_n(B_\infty^N,\ell_{q_1,q_2}^{s,b})
\ge b^{1/q_2-1/q_1} d_n(B_\infty^N,\ell_{q_1,q_1}^{s,b}) \ge b^{1/q_2-1/q_1}
\eps N^{1/q_1} = \eps s^{1/q_1}b^{1/q_2}.
$$
The subcase $q_1>q_2$ is analogous.

Case b) $p_1<q_1\le 2$, $p_2<q_2\le 2$; then $d_0=1$. We have
$$
    d_n(B_{p_1,p_2}^{s,b},\ell_{q_1,q_2}^{s,b}) \ge d_n(B_1^N,\ell_2^N) =
    (1-n/N)^{1/2} \ge \eps^{1/2}.
$$
Note that this also works for $p_1=q_1\le 2$ or $p_2=q_2\le 2$.

Case c) $p_1<q_1\le 2$, $p_2\ge q_2$. We have
$d_0=b^{1/q_2-1/p_2}$. 
    We replace the pair $(p_1,q_1)$ by
    $(1,2)$ (it does not affect $d_0$ and the width $d_n$ can only
    decrease). Next, we use the inclusion $B_{1,p_2}^{s,b}\supset
    b^{-1/p_2}B_{1,\infty}^{s,b}$, the bound $\|x\|_{2,1} \le
    b^{1-1/q_2}\|x\|_{2,q_2}$ and reduce the lower bound to the
    inequality~\eqref{oct_l21}:
    \begin{multline*}
    d_n(B_{p_1,p_2}^{s,b},\ell_{q_1,q_2}^{s,b})
    \ge d_n(B_{1,p_2}^{s,b},\ell_{2,q_2}^{s,b})
    \ge b^{-1/p_2} d_n(B_{1,\infty}^{s,b},\ell_{2,q_2}^{s,b}) \ge \\
    \ge b^{-1/p_2} b^{-1+1/q_2}d_n(B_{1,\infty}^{s,b},\ell_{2,1}^{s,b})
    \ge c(\eps) b^{1/q_2-1/p_2}.
    \end{multline*}

    Case d) $p_1\ge q_1$, $p_2 < q_2 \le 2$;  $d_0 =
    s^{1/q_1-1/p_1}$. We
    have to consider several sub-cases.
 
Case d1) $q_1\ge q_2$. We use Statement~\ref{stm_mixed_lp}:
$$
    d_n(B_{p_1,p_2}^{s,b},\ell_{q_1,q_2}^{s,b})
    \ge s^{1/q_1-1/q_2} d_n(B_{p_1,p_2}^{s,b},\ell_{q_2}^N)
    \ge c(q_2,\eps) s^{1/q_1-1/p_1}.
$$

    Case d2) $p_1=q_1 \le 2$. This case is covered by
    the same argument as in b).
    
    If the parameters do not fall into the subcases d1) or d2), it means that
    $q_1<q_2\le 2$, $q_1<p_1$, $p_2<q_2\le 2$ so we are in the exceptional case.

\paragraph{B) Upper bounds.}
    We have to prove that for all other tuples of $(p_1,p_2,q_1,q_2)$ we do not have rigidity. 
        Note that the necessary condition for the rigidity in mixed norm is that
    both $B_{p_1}^s$ in $\ell_{q_1}^s$ and
    $B_{p_2}^b$ in $\ell_{q_2}^b$ are rigid.
    We mentioned (see~\eqref{lp_ball_nonrigid}, \eqref{lp_ball_rigidity}) that $B_p^N$ is rigid in $\ell_q^N$ exactly when
    $q\le\max\{p,2\}$. Hence the conditions (i) and (ii) in A) are, of course,
    necessary and it is enough to consider only the exceptional case: $p_1>q_1$, $p_2<q_2\le 2$ and $q_1<q_2$. We remark that 
    $d_0=s^{1/q_1-1/p_1}$.
   
    Our construction is the generalization of the example
    $d_n(B_{\infty,1}^{s,s},\ell_{2,1})$. Instead of the transposition,
    we construct more complicated operator that ``distributes''
    coordinates of $x$ among miscellaneous columns in order to minimize the
    mixed norm.

    Some preparations are required. Let us use the notation
    $[N]:=\{1,2,\ldots,N\}$. We identify vectors in $\mathbb{R}^N$, $N=sb$, with
    $s\times b$ matrices and use both terms (vector or matrix) later on; the $j$-th block $x[j]$ becomes the $j$-th column ($1\le j\le b$). 

    Let $\Gamma=(\Gamma_1,\ldots,\Gamma_m)$  be a partition of the set
    $[s]\times[b]$:
    $$
    [s]\times[b] = \Gamma_1\sqcup\Gamma_2\sqcup\cdots\sqcup\Gamma_m.
    $$
    We call $\Gamma$ a $(m,r,l)$-partition if the following conditions are satisfied:
\begin{itemize}
    \item[(i)] $|\Gamma_k|\le r$ for $k=1,\ldots,m$;
    \item[(ii)] any column $\mathcal{C}^j := [s]\times\{j\}$ contains at
        most one point from each set:
        $|\mathcal{C}^j\cap\Gamma_k|\le1$, $j=1,\ldots,b$, $k=1,\ldots,m$;
    \item[(iii)] for any two columns only $l$ sets can intersect both of them:
        $$
        |\{k\colon\mathcal{C}^{j_1}\cap\Gamma_k\ne\varnothing,\;\mathcal{C}^{j_2}\cap\Gamma_k\ne\varnothing\}|\le
        l,$$ for any $1\le j_1<j_2\le b$.
\end{itemize}

    Note that if there exists a $(m,r,l)$-partition of $[s']\times[b']$, then there is a
    $(m,r,l)$-partition of $[s]\times[b]$ for any $s\le s'$, $b\le b'$. Indeed,
    the partition of a larger set induces the partition of a smaller set.

\begin{lemma}
    \label{lem_partition_from_sets}
    Suppose that sets $A_1,\ldots,A_m\subset[b]$ have cardinality at most $r$
    and satisfy the conditions:
    1) each point $j\in[b]$ lies in at least $s$ sets $A_k$;
    2) any pair $\{j_1,j_2\}\subset[b]$, $j_1\ne j_2$, is contained in at most $l$ sets $A_k$.
    Then there exists a $(m,r,l)$-partition of $[s]\times[b]$.
\end{lemma}

\begin{proof}
    Let us construct $\Gamma_1,\ldots,\Gamma_m$; we start with empty sets.
    For any $j\in[b]$ we take $s$ sets
    $A_{k_1},\ldots,A_{k_s}$ containing $j$.
    We add the point $(1,j)$ to $\Gamma_{k_1}$, the point $(2,j)$ to
    $\Gamma_{k_2}$, etc, $(s,j)$ to $\Gamma_{k_s}$. The property (ii) is not
    violated. If we do this for all $j$, we will fill the rectangle
    $[s]\times[b]$.
    Note that $\mathcal C^j\cap\Gamma_k\ne\varnothing$ only if
    $j\in A_k$; this implies (i) and (iii) and we are done.
(Some of the  sets $\Gamma_k$ may remain empty.)
\end{proof}

In order to construct regular set systems it is convenient to use combinatorial
designs. Necessary definitions and basic results can be found
in~\cite[Chs.1,5]{St04}.

Recall that a $(b,r,l)$-design is a sequence of sets
$A_1,\ldots,A_m\subset[b]$, such that $|A_k|=r$
for any $k$ and  any pair $\{j_1,j_2\}\subset[b]$, $j_1\ne j_2$, is contained
in exactly $l$ sets $A_k$.
It is well--known that any point
$j\in[b]$ lies in $l(b-1)/(r-1)$ sets of the design; moreover, $m=lb(b-1)/(r^2-r)$.
Any $(b,r,l)$-design produces a $(b,r,lh)$-design
($h\in\mathbb N$): just repeat each set $h$ times.
If $r$ is a prime power and $d\in\mathbb N$, $d\ge2$, then there exists a
$(r^d,r,1)$-design; it is given by the collection  of all affine lines in
$\mathbb{F}_r^d$.

\begin{lemma}
    \label{lem_good_partition}
    For any numbers $s,b,d\ge 2$, $s\ge b$, there exists a $(m,r,l)$-partition of
    $[s]\times[b]$ with the parameters
    $$
    b^{1/d} \le r\le 2b^{1/d},
    \quad m\le C_1(d)sb/r,
    \quad l \le C_2(d)sr/b.
    $$
\end{lemma}

\begin{proof}
    Suppose that $b=2^{ud}$ with $u\in\mathbb N$.
    Let $r=2^u$, then $b=r^d$ and a
    $(b,r,1)$-design exists. Each point of $[b]$ lies in $(b-1)/(r-1)$
    design sets. Let $l$ be the minimal number such that $l(b-1)/(r-1)\ge s$.
    The $(b,r,l)$-design obtained by the repetition of each set $l$ times,
    satisfies the conditions of Lemma~\ref{lem_partition_from_sets} and produces a
    $(m,r,l)$-partition with the parameters
    $$
    r=b^{1/d},
    \quad l = \left\lceil\frac{s(r-1)}{b-1}\right\rceil,
    \quad m=\frac{lb(b-1)}{r(r-1)}.
    $$
    The required inequalities are satisfied. 

    The general case follows, since we have  $2^{-d}b'<b\le b'$,
    $b'=2^{ud}$, with an appropriate choice of  $u$.
\end{proof}

With a $(m,r,l)$-partition $[s]\times[b]=\Gamma_1\sqcup\ldots\sqcup\Gamma_m$ one
can associate the
$m$-dimensional space $Q_\Gamma\subset\R^{s\times b}$ of vectors that are
constant on each of the sets $\Gamma_k$. Define the linear operator
$D_\Gamma\colon\mathbb R^{s\times b}\to Q_\Gamma$:
$$
D_\Gamma(e_{i,j}) = \sum_{(u,v)\in\Gamma_k}e_{u,v},\quad\mbox{where $\Gamma_k\ni(i,j)$}.
$$

\begin{lemma}
    \label{lem_operator_bound}
    Let $\Gamma$ be a $(m,r,l)$-partition of $[s]\times[b]$; then for any
    matrix
    $x\in\R^{s\times b}$ whose support is located in one column, and any
    $p,q_1,q_2\in[1,+\infty]$, we have
    $$
    \|x-D_\Gamma x\|_{q_1,q_2} \le l^{(1/q_1-1/p)_+}b^{(1/q_2-1/p)_+}(r-1)^{1/p}\|x\|_p.
    $$
\end{lemma}

\begin{proof}
    Denote $z:=x-D_\Gamma x$; the crucial property of $z$ is that it has at most
    $l$ nonzero elements in each column (this follows from the property (iii) of
    the partition).
    Hence, $\|z[j]\|_{q_1} \le l^{(1/q_1-1/p)_+}\|z[j]\|_p$ and
    $$
    \|z\|_{q_1,q_2} \le b^{(1/q_2-1/p)_+}\|z\|_{q_1,p}
    \le b^{(1/q_2-1/p)_+}l^{(1/q_1-1/p)_+}\|z\|_p.
    $$
    Any nonzero coordinate of $x$ is found among the coordinates of $z$ at most
    $r-1$ times; hence $\|z\|_p^p \le (r-1)\|x\|_p^p$ and we finish the proof.
\end{proof}

We return to the proof of Theorem. Let us start with the case
$s\ge b$.
Pick some $d$ (we will specify this choice later) and apply
Lemmas~\ref{lem_good_partition} and~\ref{lem_operator_bound}. We see  that
for any $x$ with support in one column, we have
\begin{multline*}
\|x-D_\Gamma x\|_{q_1,q_2}
\le l^{(1/q_1-1/p_1)_+}b^{(1/q_2-1/p_1)_+}r^{1/p_1}\|x\|_{p_1} \le \\
    \le C(d)(sr/b)^{1/q_1-1/p_1} b^{(1/q_2-1/p_1)_+} r^{1/p_1} \|x\|_{p_1}
    = C(d) s^{1/q_1-1/p_1}b^{-\alpha} r^{1/q_1} \|x\|_{p_1},
\end{multline*}
where $\alpha := (\frac1{q_1}-\frac1{p_1}) - (\frac1{q_2}-\frac1{p_1})_+$. Note
that as $q_1<q_2$ and $q_1<p_1$, we have $\alpha>0$, hence
\begin{equation}
    \label{one_column_bound}
\|x-D_\Gamma x\|_{q_1,q_2} \le C(d) s^{1/q_1-1/p_1}b^{-\alpha/2} \|x\|_{p_1},
\end{equation}
if we choose $d$ large enough. (If $p_2=1$, this would be enough to bound the
width.)

Given $y\in\R^b$ and $k\le b$, we consider the best $k$--term approximation  
$$
\sigma_k(y)_q := \min\{\|y-z\|_q\colon \#\supp z\le k\}.
$$
There is a well--known inequality (see, say, \cite{Tem18}, Lemma 7.6.6):
$$
\sigma_{k-1}(y)_q \le k^{-(1/p-1/q)}\|y\|_p,\quad 1\le p<q\le \infty.
$$

Fix $x\in B_{p_1,p_2}^{s,b}$. Then the vector
$y=(\|x[1]\|_{p_1},\ldots,\|x[b]\|_{p_1})$ has norm $\|y\|_{p_2}\le 1$. Hence,
$$
\sigma_{k-1}(y)_{q_2}\le k^{-(1/p_2-1/q_2)}\|y\|_{p_2} \le k^{-(1/p_2-1/q_2)} =:
\delta,
$$
we shall choose $k$ later. So, there is a set of indices
$\Lambda$, $|\Lambda|<k$, such that
$\|y^\Lambda-y\|_{q_2}\le\delta$, where $y^\Lambda\in \mathbb{R}^b$ coincides
with $y$ on $\Lambda$ and is zero outside $\Lambda$.
Therefore, $\|x-x^\Lambda\|_{p_1,q_2}\le\delta$, where
$x^\Lambda = \sum_{j\in\Lambda} x^j$,
$x^j\in \R^N $ is the matrix  with $j$-th column  equal to $x[j]$ and zeros at all other places.
So, according to~\eqref{one_column_bound},
\begin{multline*}
    \|x^\Lambda - D_\Gamma x^\Lambda\|_{q_1,q_2} \le
    \sum_{j\in\Lambda}\|x^j-D_\Gamma x^j\|_{q_1,q_2} \le
    C(d) s^{1/q_1-1/p_1}b^{-\alpha/2} \sum_{j\in\Lambda} \|x^j\|_{p_1} \le \\
    \le C(d) k s^{1/q_1-1/p_1}b^{-\alpha/2}.
\end{multline*}
Hence the vector $x$ is approximated by  vector
$D_\Gamma x^\Lambda\in Q_\Gamma$ with the error
$$
\|x-D_\Gamma x^\Lambda\|_{q_1,q_2} \le \|x-x^\Lambda\|_{q_1,q_2} +
\|x^\Lambda-D_\Gamma x^\Lambda\|_{q_1,q_2} \le \delta s^{1/q_1-1/p_1} + C(d)ks^{1/q_1-1/p_1}b^{-\alpha/2}.
$$
So,
$$
d_{m}(B_{p_1,p_2}^{s,b},\ell_{q_1,q_2}^{s,b})
\le \delta s^{1/q_1-1/p_1} + C(d)ks^{1/q_1-1/p_1}b^{-\alpha/2}.
$$
Note that $m\le C(d)sb/r\le C(d)Nb^{-1/d}$.
It remains to take $k\asymp b^{\alpha/4}$ to get the required upper bound.

Finally, we have to deal with the case $s < b$.
Let us group all the blocks into $\asymp b/s$ groups of $\asymp s$ blocks. In
each group we take the subspace of dimension $\asymp s^{2-\gamma}$ that provides
good approximation. Then the sum of these subspaces provides the required
approximation in the whole space (here we use that $p_2<q_2$).

\paragraph{Acknowledgements.} 
The authors thank A.A.~Vasil'yeva for fruitful discussions.

\end{document}